\documentclass[10pt]{mcom-l}
\usepackage{amsthm,amssymb,amsfonts,amsmath,epsfig,hyperref}
\usepackage{mathrsfs}
\usepackage{graphics}   
\usepackage{cite}            
\usepackage{color}     

\newtheorem{theorem}{Theorem}[section]
\newtheorem{lemma}[theorem]{Lemma}
\theoremstyle{remark}

\definecolor{darkred}{rgb}{.7,0,0}

\definecolor{green}{rgb}{0,0.7,0}

\renewcommand{\theequation}{\thesection.\arabic{equation}}
\allowdisplaybreaks

\def\R{\mathbb{R}}

\def\d{\,{\rm d}}
\def\le{\leqslant}
\def\ge{\geqslant}
\def\Omega{\varOmega}
\def\Delta{\varDelta}
\def\Lambda{\varLambda}

\title{ 
Weak discrete maximum principle of\\ 
finite element methods in convex polyhedra } 

\author[Dmitriy Leykekhman]{Dmitriy Leykekhman}
\address{Department of Mathematics, University of Connecticut, Storrs, CT 06269, USA. } 
\email {\href{mailto:dmitriy.leykekhman@uconn.edu}{dmitriy.leykekhman@uconn.edu}}

\author[Buyang Li]{Buyang Li}
\address{Department of Applied Mathematics, 
The Hong Kong Polytechnic University, Hung Hom, Hong Kong.} 
\email {\href{mailto:buyang.li@polyu.edu.hk}{buyang.li{\it @}polyu.edu.hk}}

\thanks{This work is partially supported by NSF DMS-1913133 and a Hong Kong RGC grant (project no. 15300519). }

\begin{document}

\maketitle

\begin{abstract}
\small  
We prove that the Galerkin finite element solution $u_h$ of the Laplace equation in a convex polyhedron $\varOmega$, with a quasi-uniform tetrahedral partition of the domain and with finite elements of polynomial degree $r\ge 1$, satisfies the following weak maximum principle:
\begin{align*}
\left\|u_{h}\right\|_{L^{\infty}(\varOmega)} \le 
C\left\|u_{h}\right\|_{L^{\infty}(\partial \varOmega)} , 
\end{align*}
with a constant $C$ independent of the mesh size $h$. By using this result, we show that the Ritz projection operator $R_h$ is stable in $L^\infty$ norm uniformly in $h$ for $r\geq 2$, i.e.
\begin{align*}
\|R_hu\|_{L^{\infty}(\varOmega)} \le 
C\|u\|_{L^{\infty}(\varOmega)} .
\end{align*}
Thus we remove a logarithmic factor appearing in the previous results for convex polyhedral domains. 
\end{abstract}

\pagestyle{myheadings}
\markboth{}{}

\section{Introduction}
\setcounter{equation}{0}

Let $S_h$ be a finite element space of Lagrange elements of degree $r\ge 1$ subject to a quasi-uniform tetrahedral partition $\mathfrak{T}$ of a convex polyhedron $\Omega\subset\R^3$, where $h$ denotes the mesh size of the tetrahedral partition, and quasi-uniformity means that
$$
\rho_\tau \ge c h \quad\forall\,\tau\in\mathfrak{T},
$$
with $\rho_\tau$ denoting the radius of the largest ball inscribed in the tetrahedron $\tau\in\mathfrak{T}$. 

Let $\mathring S_h$ be the subspace of $S_h$ consisting of functions with zero boundary values. A function $u_h\in S_h$ is called a discrete harmonic if it satisfies  
\begin{align}\label{discrete-harmonic}
(\nabla u_h,\nabla \chi_h)=0\quad\forall\,\chi_h\in\mathring S_h.
\end{align}
In this article, we establish the following result, which we call weak maximum principle of finite element methods (for higher order equations it is often called Agmon--Miranda maximum principle).  
\begin{theorem}\label{THM1}
A discrete harmonic function $u_h$ satisfies the following estimate:
\begin{align}\label{weak-max-principle}
\left\|u_{h}\right\|_{L^{\infty}(\Omega)} \leqslant C\left\|u_{h}\right\|_{L^{\infty}(\partial \Omega)} ,
\end{align}
where the constant $C$ is independent of the mesh size $h$. 
\end{theorem}

As an application of the weak maximum principle, we show that the Ritz projection $R_h:H^1_0(\Omega)\rightarrow \mathring S_h$ defined by
$$
(\nabla(u-R_hu),\nabla v_h)=0\quad\forall\, v_h\in\mathring S_h
$$
is stable in $L^\infty$ norm for  finite elements of degree $r\ge 2$, i.e. 
\begin{equation*}
\|R_hu\|_{L^\infty(\Omega)} \le C\|u\|_{L^\infty(\Omega)}
\quad\forall\, u\in H^1_0(\Omega)\cap L^\infty(\Omega). 
\end{equation*}
Although this result is well-known for smooth domains \cite{Schatz-Wahlbin-1977, Schatz-Wahlbin-1982}, for convex polyhedral domains the result was available only with an additional logarithmic factor \cite[Theorem 12]{Leykekhman_Vexler_2016}.

In the finite element literature, the ``strict" discrete maximum principle
$$
\|u_{h}\|_{L^{\infty}(\Omega)} \leqslant 
\|u_{h}\|_{L^{\infty}(\partial \Omega)} 
$$
i.e., with $C=1$ in \eqref{weak-max-principle}, has attracted a lot of attention; see \cite{Ciarlet_1970, CiarletRaviart_1973, Ruas Santos_1982, Vanselow_2001, WangZhang_2012}, to mention a few. However, the sufficient conditions for the strict discrete maximum principle often put serious restrictions on the geometry of the mesh. 
For piecewise linear elements in two-dimensions, the strict discrete maximum principle generally requires the angles of the triangles to be less than $\pi/2$, or the sum of opposite angles of the triangles that share an edge to be less than $\pi$ (for example, see \cite[\textsection 5]{WangZhang_2012}), though these conditions are not necessary away from the boundary \cite{2004-Draganescu}. For quadratic elements in two dimensions, discrete maximum principe holds only  for equilateral triangles \cite{Hohn_Mittelmann_1981}. The situation in three dimensions is more complicated \cite{Brandts_Korotov_Krizek_2009, Korotov_Krizek_2001, Korotov_Krizek_Pekka_2001, Xu_Zikatanov_1999}, essentially  it is hard to guarantee the discrete maximum principe even for piecewise linear elements.

A different approach was taken in the work of Schatz \cite{Schatz-1980}, who proved that a weak maximum principle in the sense of \eqref{weak-max-principle} holds for a wide class of finite elements on general quasi-uniform triangulation of any two dimensional polygonal domain. The weak maximum principle was used to established the stability of the Ritz projection in $L^\infty$ and $W^{1,\infty}$ norms for two-dimensional polygons. 
Such $L^\infty$- and $W^{1,\infty}$-stability results have a wide range of applications, for example to pointwise error estimates of finite element methods for parabolic problems \cite{2019-Li, Kashiwabara-Kemmochi, Leykekhman_Vexler_2016b},  Stokes systems \cite{Behringer_Leykekhman_Vexler_2019}, nonlinear problems \cite{Frehse_Rannacher_1978, Demlow_2006, Meinder_Vexler_2016b}, obstacle problems \cite{Christof_2017}, optimal control problems \cite{Apel_Rosch_Sirch_2009, Apel_Winkler_Pfefferer_2018}, to name a few.
As far as we know, \cite{Schatz-1980} is the only paper that establishes weak maximum principle and $L^\infty$ stability estimate (without the logarithmic factor) for the Ritz projection on nonsmooth domains.

In three dimensions the situation is less satisfactory. The stability of the Ritz projection in $L^\infty$ and $W^{1,\infty}$ norms is available on smooth domains \cite{Schatz-Wahlbin-1977, Schatz-Wahlbin-1982} and convex polyhedral domains \cite{Guzman_Leykekhman_Rossman_Schatz_2009,Leykekhman_Vexler_2016}. However, on convex polyhedral domains in \cite{Leykekhman_Vexler_2016}, the $L^\infty$-stability constant depends logarithmically on the mesh size $h$, and it is not obvious how the logarithmic factor can be removed there. There are no results on the weak maximum principles in  three dimensions even on smooth domains or convex polyhedra. 
The objective of this paper is to close this gap for convex polyhedral domains. 
In order to obtain the result, we have to modify the argument in \cite{Schatz-1980} by extending the arguments to $L^p$ norm for some $1<p<2$. This constitutes the main technical difficulty in the analysis of the paper. The mere adaptation of  the  $L^2$-norm based argument used in \cite{Schatz-1980} for  convex polyhedral domains, would yield a logarithmic factor. 
Unfortunately, the current analysis does not allow us to extend the results to nonconvex polyhedral domains or graded meshes. These would be the subject of future research. 

The paper is organized as follows. In section \ref{sec: preliminary} we state some preliminary results that we use later in our arguments. In section \ref{sec: start of the proof}, we reduce the proof of the weak discrete maximum principle to a specific error estimate. Section  \ref{sec: main proof} is devoted to the proof of this estimate, which constitutes the main technical part of the paper. Finally, section \ref{sec: stability Ritz}, gives an application of the weak discrete maximum principle to showing the stability of the Ritz projection in $L^\infty$ norm uniformly in $h$ for higher order elements.

In the rest of this article, we denote by $C$ a generic positive constant, which may be different at different occurrences but will be independent of the mesh size $h$. 

\section{Preliminary results} \label{sec: preliminary}
 
In this section, we present several well-known results that are used in our analysis. 
First result concerns global regularity of the weak solution  $v\in H^1_0(\Omega)$ to the problem 
\begin{align}\label{weak-PDE} 
(\nabla v,\nabla\chi)=(f,\chi) 
\quad\forall\,\chi\in H^1_0(\Omega) .
\end{align} 
 On the general convex domains we naturally have  
 the $H^2$ regularity  (cf. \cite{Grisvard_1985}). 
However, on convex polyhedral domains, we have the following sharper $W^{2,p}(\Omega)$ regularity result (cf. \cite[Corollary 3.12]{1992-Dauge}). 
\begin{lemma}\label{lemma: W2p regularity}
Let $\Omega$ be a convex polyhedron. Then there exists a constant $p_0>2$ depending on $\Omega$ such that for any $1<p<p_0$ and $f\in L^{p}(\Omega)$, the solution $v$ of \eqref{weak-PDE} is in $W^{2,p}(\Omega)$ and 
$$
\|v\|_{W^{2,p}(\Omega)}
\le
C\|f\|_{L^p(\Omega)}.
$$
\end{lemma}

For any point $x^*\in \overline\Omega$ we denote $S_d(x^*)=\{x\in\Omega:|x-x^*|<d\}$. The following result, which is a version of the Poincare inequality, is an extension of Lemma 1.1 in \cite{Schatz-1980}, which was established in two dimensions for $p=2$.

\begin{lemma}\label{W1p0-d}
Let $1<p<\infty$. 
If $\chi\in W^{1,p'}_0(\Omega)$ and $x^*\in \partial\Omega$, then  
$$
\|\chi\|_{L^{p}(S_{d_*}(x^*))} 
\le
Cd_*\|\nabla\chi\|_{L^{p}(\Omega)}.
$$
\end{lemma}
\begin{proof}
Similarly to \cite[Lemma~1.1]{Schatz-1980}, we consider $\chi\in C_0^\infty(\Omega)$ and extend $\chi$ by zero outside $\Omega$. By denoting $x^*=(x^*_{1},x^*_{2},x^*_{3})$ and using the spherical coordinates centered at $x^*$, we define  
$$
\widetilde\chi(\rho,\varphi,\theta)
=
\chi(x^*_{1}+\rho\sin(\varphi)\cos(\theta),
x^*_{2}+\rho\sin(\varphi)\sin(\theta),
x^*_{3}+\rho\cos(\varphi) ) ,
$$
for $0\le\rho\le d_*$, and some $\varphi\in[0,\pi]$ and $\theta\in[0,2\pi]$. Since $\chi=0$ on $\partial\Omega$, there exists $\theta^* \in[0,2\pi]$ such that $\widetilde\chi(\rho,\varphi,\theta^*)=0$. Therefore, 
$$
|\widetilde\chi(\rho,\varphi,\theta)|
=\bigg|\int_{\theta^*}^{\theta}
\partial_\theta \widetilde\chi(\rho,\varphi,\theta')\d\theta'\bigg|
\le
\int_{0}^{2\pi}
|\partial_\theta \widetilde\chi(\rho,\varphi,\theta')|\d\theta' .
$$
From the chain rule, we have
$$
\partial_\theta\widetilde \chi(\rho,\varphi,\theta)
=- {\partial_{x_1}\widetilde\chi}(\rho,\varphi,\theta)\,  
\rho \sin(\varphi)\sin(\theta)
+{\partial_{x_2}\widetilde\chi}(\rho,\varphi,\theta)\,  
\rho \sin(\varphi)\cos(\theta) .
$$
As a result, by H\"older's inequality, we obtain
$$
|\widetilde\chi(\rho,\varphi,\theta)|^p
\le
C\int_{0}^{2\pi}
|\partial_\theta \widetilde\chi(\rho,\varphi,\theta')|^p\d\theta'
\le
C\int_{0}^{2\pi}\rho^p
| {\nabla \widetilde\chi}(\rho,\varphi,\theta')|^p\d\theta' .
$$
Therefore, 
\begin{align*}
\int_{S_{d_*}(x^*)} |\chi(x)|^p\d x
&= \int_0^{d_*}\int_0^{ \pi}\int_0^{2\pi} |\widetilde\chi(\rho,\varphi,\theta)|^p \rho^2\sin(\varphi)\d\theta\d\varphi\d\rho \\
&\le
C\int_0^{d_*}\int_0^{ \pi}\int_0^{2\pi}\bigg(\int_{0}^{2\pi}
\rho^p|{\nabla\widetilde\chi}(\rho,\varphi,\theta')|^p\d\theta' \bigg) \rho^2\sin(\varphi)\d\theta \d\varphi\d\rho \\
&\le
Cd_*^p\int_0^{d_*}\int_0^{ \pi}\int_0^{2\pi}|{\nabla\widetilde\chi}(\rho,\varphi,\theta')|^p \rho^2 \sin(\varphi) \d\theta' \d\varphi\d\rho \\
&=
Cd_*^p \int_{S_{d_*}(x^*)} |\nabla\chi(x)|^p\d x.
\end{align*}
This proves the desired result. 
\end{proof}

The next result addresses the problem \eqref{weak-PDE} when the source function $f$ is supported in some part of $\Omega$. It establishes the stability of the solution in $W^{1,p}$ norm and traces the dependence of the stability constant on the diameter of the support. The corresponding result in \cite{Schatz-1980} is the equation (1.6) therein, which was established for $p=2$ in two dimensions. In our situation we need it for larger range of $p$.

\begin{lemma}\label{W1p-d}
For any bounded Lipschitz domain $\Omega$, there exist  positive constants $\alpha\in(0,\mbox{$\frac12$})$ and $C$ (depending on $\Omega$) such that for $\frac32-\alpha\le p\le 3+\alpha$ and $f\in L^p(\Omega)$ with ${\rm supp}(f)\subset S_{d_*}(x_0)$ and ${\rm dist}(x_0,\partial\Omega)\le d_*$, the solution of \eqref{weak-PDE} 
satisfies
$$
\|v\|_{W^{1,p}(\Omega)}
\le
Cd_*\|f\|_{L^p(\Omega)} .
$$
\end{lemma}
\begin{proof}
If ${\rm dist}(x_0,\partial\Omega)\le d_*$, then $S_{d_*}(x_0)\subset S_{2d_*}(\bar x_0)$ for some $\bar x_0\in\partial\Omega$. 
For any $\chi\in W^{1,p'}_0(\Omega)$, there holds
\begin{align*}
|(\nabla v,\nabla\chi)|
=|(f,\chi)|
&\le
\|f\|_{L^p(S_{d_*}(x_0))}
\|\chi\|_{L^{p'}(S_{d_*}(x_0))} \\
&
\le
\|f\|_{L^p(S_{d_*}(x_0))}
\|\chi\|_{L^{p'}(S_{2d_*}(\bar x_0))} \\
&\le
Cd_*\|f\|_{L^p(\Omega)}
\|\nabla\chi\|_{L^{p'}(\Omega)} ,
\end{align*}
where in the last step we have used Lemma \ref{W1p0-d}. 

For $\vec{w}\in C_0^\infty(\Omega)^3$, we let $\chi\in H^{2}(\Omega)\cap H^1_0(\Omega)\hookrightarrow W^{1,p'}_0(\Omega)$ (for $\frac32-\alpha \le p\le 3+\alpha$) be the solution of 
$$
\left\{
\begin{aligned}
\Delta \chi&=\nabla\cdot \vec{w} &&\mbox{in}\,\,\,\Omega \\
\chi&=0 &&\mbox{on}\,\,\,\partial\Omega.
\end{aligned}
\right. 
$$
The solution $\chi$ defined above satisfies 
$$
\nabla\cdot(\vec{w}-\nabla\chi) = 0 ,
$$
and, according to \cite[Theorem B]{Jerison-Kenig-1995}, there exists a  constant $\alpha\in(0,\frac12)$ such that 
$$
\|\nabla\chi\|_{L^{p'}(\Omega)}
\le
C\|\vec{w}\|_{L^{p'}(\Omega)} 
\quad\mbox{for}\,\,\, \mbox{$\frac32$}-\alpha\le p\le 3+\alpha. 
$$
By using these properties, we have 
\begin{align*}
|(\nabla v,\vec{w})|
=
|(\nabla v,\nabla\chi)|
&\le
Cd_*\|f\|_{L^p(\Omega)}
\|\nabla\chi\|_{L^{p'}(\Omega)}
\le
Cd_*\|f\|_{L^p(\Omega)}
\|\vec{w}\|_{L^{p'}(\Omega)} . 
\end{align*}
Since $C_0^\infty(\Omega)^3$ is dense in $L^{p'}(\Omega)^3$ and the estimate above holds for all $\vec{w}\in C_0^\infty(\Omega)^3$, the duality pairing between $L^{p}(\Omega)^3$ and $L^{p'}(\Omega)^3$ implies the desired result. 
\end{proof}

The next lemma concerns basic properties of harmonic functions on  convex domains. The result is essentially the same as  in \cite[Lemma 8.3]{Schatz-Wahlbin-1978}. 
\begin{lemma}\label{lemma: Caccipoli}
Let $D$ and $D_d$ be two subdomains satisfying $D \subset D_d\subset \Omega$, with 
$$
D_d=\{x\in \Omega:
dist(x, D)\le d\} ,
$$ 
where $d$ is a positive constant. 
If $v\in H^1_0(\Omega)$ and $v$ is harmonic on $D_d$, i.e.
$$
(\nabla v, \nabla w)=0, \quad \forall w\in H^1_0(D_d) , 
$$
then the following estimates hold: 
\begin{subequations}          \label{eq: Caccipoli}
\begin{align}
|v|_{H^2(D)}&\le Cd^{-1}\|v\|_{H^1(D_d)} \label{eq: Caccipoli 1},\\
\| v\|_{H^1(D)}&\le Cd^{-1}\|v\|_{L_2(D_d)}\label{eq: Caccipoli 2}.
\end{align}
\end{subequations}
\end{lemma}
Finally, we need the best approximation property of the Ritz projection in $W^{1,p}$ norm.
In \cite{Guzman_Leykekhman_Rossman_Schatz_2009}, the best approximation property of the Ritz projection in $W^{1,\infty}$ norm was established on convex polyhedral domains. Together with the standard best approximation property  in $H^{1}$ norm
we obtain
\begin{equation}\label{eq: best in W1p}
\|v-R_h v\|_{W^{1,p}(\Omega)}
\le  C\min_{\chi\in \mathring S_h} \|v-\chi\|_{W^{1,p}(\Omega)}  
{\quad\forall\, v\in H^1_0(\Omega)\cap W^{1,p}(\Omega),}
\end{equation}
for any $2\le p\le \infty$. Extension of the above result to $1< p\le \infty$ follows by duality (cf. \cite[\textsection 8.5]{Brenner_Scott}).
These can be summarized as below. 
\begin{lemma}\label{Lemma:Ritz-error-W1p}
On a convex polyhedron $\Omega$, the following estimate holds for any fixed $p\in(1,\infty]$: 
\begin{align*} 
\|v-R_hv\|_{W^{1,p}(\Omega)}
\le 
Ch \|v\|_{W^{2,p}(\Omega)} \quad\forall\,
v\in H^1_0(\Omega)\cap W^{2,p}(\Omega) . 
\end{align*}
\end{lemma}

In sections \ref{sec: start of the proof}--\ref{sec: main proof}, we would use several results from \cite{1974-Nitsche-Schatz,Schatz-1980,Schatz-Wahlbin-1977}. Some of these results were stated therein for sufficiently small mesh size $h$ under certain hypothesis on the triangulation. Since, we concentrate on the Lagrange elements, all the hypotheses in \cite{Schatz-Wahlbin-1977} are trivially satisfied and we assume that our mesh size $h$ is sufficiently small, say $h\le h_0$ for some constant $h_0$, so these results hold. 

\section{Basic estimates} \label{sec: start of the proof}
\setcounter{equation}{0}

In this section, we derive some estimates we require to establish one of our key results, Theorem \ref{THM1}. This part of the argument up to \eqref{to-prove} is analogous to the first part of the proof of \cite[Theorem 1]{Schatz-1980} up to equation (3.10). The dyadic decomposition part is also similar. The essential difference lies in the duality argument in section \ref{sec: main proof}, after the equation \eqref{remains-4}. 

In \cite[Corollary 5.1]{Schatz-Wahlbin-1977}, the following interior error estimate was established 
$$
\|u-u_h\|_{L^\infty(\Omega_1)}\le Ch^l|\ln{h}|^{\bar{r}}|u|_{W^{l,\infty}(\Omega_2)}+Cd^{-3/q-p}\|u-u_h\|_{W^{-p,q}(\Omega_2)},
$$
for $0\le l\le r$, where $\bar{r}=1$ for $r=1$, $\bar{r}=0$ for $r\ge 2$ and  $\Omega_1\subset\subset\Omega_2\subset\subset\Omega$, with $\operatorname{dist}(\Omega_1,\partial \Omega_2)\ge d\ge kh$ and $\operatorname{dist}(\Omega_2,\partial \Omega)\ge d\ge kh$.
Choosing $u=0$, $p=0$ and $q=2$ in the above estimate,  we obtain that there exists a constant $C$ independent of $h$ such that  
\begin{align}\label{interior-Linfty}
\|u_{h}\|_{L^\infty(\Omega_1)}
\le
C d^{-\frac32}
\|u_{h}\|_{L^{2}(\Omega_2)} . 
\end{align}
Let $x_0\in\overline\Omega$ be a point satisfying 
$$
|u_h(x_0)|=\|u_h\|_{L^\infty(\Omega)}
\quad\mbox{with}\quad d={\rm dist}(x_0,\partial\Omega) .
$$
If $d\ge 2kh$ then we can choose $\Omega_1=S_{d/2}(x_0)$ and $\Omega_2=S_{d}(x_0)$. In this case, the following interior $L^\infty$ estimate holds (cf. \cite[Corollary 5.1]{Schatz-Wahlbin-1977} and \cite[Lemma 2.1 (ii)]{Schatz-1980}):  
$$
|u_{h}(x_{0})| 
\le
C d^{-\frac32}
\|u_{h}\|_{L^{2}(S_{d}(x_{0}))} . 
$$
Otherwise, we have $d\leqslant 2kh$. In this case, the inverse inequality of finite element functions (cf. \cite[Ch.~4.5]{Brenner_Scott}) implies 
$$
|u_{h}(x_{0})| 
=
\|u_{h}\|_{L^\infty(S_{h}(x_{0}))}
\le 
C h^{-\frac32}
\|u_{h}\|_{L^{2}(S_{h}(x_{0}))}
.
$$
Hence, either for $d\ge 2kh$ or $d\le 2kh$, the following estimate holds: 
\begin{align}\label{Linfty-1}
|u_{h}(x_{0})| 
\le
C \rho^{-\frac32}
\|u_{h}\|_{L^{2}(S_{\rho}(x_{0}))},\quad\mbox{with}\quad \rho=d+2kh.
\end{align}


To estimate the term $\|u_{h}\|_{L^{2}(S_{\rho}(x_{0}))}$ on the right hand side of the inequality above, we use the following duality property: 
$$
\|u_{h}\|_{L^{2}(S_{\rho}(x_{0}))}
= \sup_{
\begin{subarray}{c}
{\rm supp}(\varphi)\subset S_{\rho}(x_{0}) \\
\|\varphi\|_{L^{2}(S_{\rho}(x_{0}))}\le 1
\end{subarray}
}|(u_{h},\varphi)| , 
$$
which implies the existence of a function $\varphi\in C_0^\infty(\Omega)$ with the following properties: 
\begin{align}\label{L2-varphi}
{\rm supp}(\varphi)\subset S_{\rho}(x_{0}),
\quad
\|\varphi\|_{L^{2}(S_{\rho}(x_{0}))}\le 1
\end{align}
and
\begin{align}\label{L2-Srho-uh}
\|u_{h}\|_{L^{2}(S_{\rho}(x_{0}))}
\le 2|(u_{h},\varphi)| . 
\end{align}
For this function $\varphi$, we define $v\in H^1_0(\Omega)$ to be the solution of 
\begin{align}\label{weak-PDE-v}
(\nabla v,\nabla\chi)=(\varphi,\chi)
\quad\forall\,\chi\in H^1_0(\Omega),
\end{align}
and let $v_h\in \mathring S_h$ be the finite element solution of 
$$
(\nabla v_h,\nabla\chi_h)=(\varphi,\chi_h)
\quad\forall\,\chi_h\in \mathring S_h. 
$$
Thus, $v_h$ is the Ritz projection of $v$ and satisfies 
\begin{align}\label{weak-form-v-vh}
(\nabla (v-v_h),\nabla\chi_h)= 0 
\quad\forall\,\chi_h\in \mathring S_h.
\end{align}
Let $u$ be the solution of the problem (in weak form) 
\begin{align}\label{weak-form-u}
\left\{\begin{aligned}
&(\nabla u,\nabla \chi)=0 &&\forall\, \chi\in H^1_0(\Omega),\\
&u=u_h &&\mbox{on}\,\,\,\partial\Omega.
\end{aligned}\right.
\end{align}
Then the continuous maximum principle of \eqref{weak-form-u} implies 
\begin{align}\label{max-principle-PDE}
\|u\|_{L^{\infty}(\Omega)} \leqslant \|u_{h}\|_{L^{\infty}(\partial \Omega)}.
\end{align}
Notice, that $u_h$ is the Ritz projection of $u$, i.e. 
$$
\left\{\begin{aligned}
&(\nabla (u-u_h),\nabla \chi_h)=0 &&\forall\, \chi_h\in \mathring S_h,\\
&u-u_h=0 &&\mbox{on}\,\,\,\partial\Omega.
\end{aligned}\right.
$$
Therefore, we have 
\begin{align}\label{L2-uh-1}
\|u_{h}\|_{L^{2}(S_{\rho}(x_{0}))} 
&\le
2|(u_{h},\varphi)| 
&&\mbox{(here we used \eqref{L2-Srho-uh})} \nonumber \\
&=2|(u_{h}-u,\varphi) + (u,\varphi)| \nonumber \\
&= 2|(\nabla(u_h-u),\nabla v) + (u,\varphi)| 
&&\mbox{(here we used \eqref{weak-PDE-v})} \nonumber\\
&=2| (\nabla u_h ,\nabla v) + (u,\varphi)| 
&&\mbox{(here we used \eqref{weak-form-u})} \nonumber\\ 
&\le 2|(\nabla u_h ,\nabla v)| + 2\|u\|_{L^\infty(\Omega)}\|\varphi\|_{L^1(\Omega)} \nonumber\\
&\le 2|(\nabla u_h ,\nabla v)| + C\rho^{\frac32} \|u_h\|_{L^\infty(\partial\Omega)}\|\varphi\|_{L^2(S_{\rho}(x_{0}))} , \end{align}
where we have used \eqref{max-principle-PDE} and the H\"older inequality in deriving the last inequality. 

To estimate $|(\nabla u_h ,\nabla v)|$, we note that 
\begin{align*}
(\nabla u_h ,\nabla v)
&=(\nabla u_h ,\nabla (v-v_h)) &&\mbox{(here we use \eqref{discrete-harmonic} and $v_h\in\mathring S_h$)} \\ 
&=(\nabla (u_h-\chi_h) ,\nabla (v-v_h)) 
&&\forall\,\chi_h\in \mathring S_h. 
\quad\mbox{(here we use \eqref{weak-form-v-vh})}.
\end{align*}
We simply choose $\chi_h$ to be equal to $u_h$ at interior nodes and $\chi_h=0$ on $\partial\Omega$; thus $u_h(x)-\chi_h(x)$ is zero when ${\rm dist}(x,\partial\Omega)\ge h$, and for any $r\geq 1$ 
$$
\|u_h-\chi_h\|_{L^\infty(\Omega)} 
\le
C \|u_h\|_{L^\infty(\partial\Omega)} . 
$$
If we define 
$$
\Lambda_h=\{x\in\Omega:{\rm dist}(x,\partial\Omega)\le h\},
$$
then using the inverse inequality,
\begin{align}\label{grad-uh-vh}
|(\nabla u_h ,\nabla v)| 
&\le \|\nabla (u_h-\chi_h)\|_{L^\infty(\Lambda_h)} \|\nabla (v-v_h)\|_{L^1(\Lambda_h)} \nonumber \\
&\le
Ch^{-1}\|u_h-\chi_h\|_{L^\infty(\Omega)} \|\nabla (v-v_h)\|_{L^1(\Lambda_h)} \nonumber \\
&\le
Ch^{-1}\|u_h\|_{L^\infty(\partial\Omega)} \|\nabla (v-v_h)\|_{L^1(\Lambda_h)} .
\end{align}
Then, substituting \eqref{L2-uh-1} and \eqref{grad-uh-vh} into \eqref{Linfty-1}, we obtain 
\begin{align}\label{Linfty-uh-2}
\|u_h\|_{L^\infty(\Omega)} 
\le
C\big(\rho^{-\frac32} h^{-1} \|\nabla (v-v_h)\|_{L^1(\Lambda_h)}
+1)
\|u_h\|_{L^\infty(\partial\Omega)} . 
\end{align}
The proof of Theorem \ref{THM1} will be completed if we establish
\begin{align}\label{to-prove}
\rho^{-\frac32} h^{-1} \|\nabla (v-v_h)\|_{L^1(\Lambda_h)}\le C, 
\end{align}
which will be accomplished in the next section. 

\section{Estimate of 
$\rho^{-\frac32} h^{-1} \|\nabla (v-v_h)\|_{L^1(\Lambda_h)}$}
\label{sec: main proof}

Let $R_0={\rm diam}(\Omega)$ and $d_j=R_02^{-j}$ for $j=0,1,2,\dots$
We define a sequence of subdomains 
$$
A_j=\{x\in\Omega:d_{j+1} \le|x-x_0|\le d_j \} , 
\quad j=0,1,2,\dots
$$
For each $j$ we denote $A_j^l$ to be a subdomain slightly larger than $A_j$, defined by 
$$
A_j^l=
A_{j-l}
\cup
\cdots
\cup
A_{j}
\cup
A_{j+1}
\cup
\cdots
\cup
A_{j+l}
\quad l=1,2,\dots
$$
Let $J=[\ln_2(R_0/8\rho)]+1$, with $[\ln_2(R_0/8\rho)]$ denoting the greatest integer not exceeding $\ln_2(R_0/8\rho)$. 
Then 
$$
2\rho\le d_{J+1}\le 4\rho  
$$
and
\begin{align}\label{volume-Aj}
{\rm measure}(A_j\cap \Lambda_h) \leq C h d_j^2 . 
\end{align}
By using these subdomains defined above, we have 
\begin{align}\label{remains-1}
&\rho^{-\frac32} h^{-1} \|\nabla (v-v_h)\|_{L^1(\Lambda_h)} 
\nonumber \\ 
&\le
\rho^{-\frac32} h^{-1} 
\bigg(\sum_{j=0}^J
\|\nabla (v-v_h)\|_{L^1(\Lambda_h\cap A_j)} 
+
\|\nabla (v-v_h)\|_{L^1(\Lambda_h\cap S_{4\rho}(x_0))} 
\bigg)
\nonumber \\
&\le
C\rho^{-\frac32} h^{-1} \sum_{j=0}^J h^{\frac12}d_j
\|\nabla (v-v_h)\|_{L^2(\Lambda_h\cap A_j)} 
\nonumber \\
&\quad\,
+
C\rho^{-\frac12} h^{-\frac12} 
\|\nabla (v-v_h)\|_{L^2(\Lambda_h\cap S_{4\rho}(x_0))},
\end{align}
where the H\"older inequality and \eqref{volume-Aj} were used in deriving the last inequality. 

Using global error estimate in $H^1$ norm, Lemma \ref{lemma: W2p regularity}
with $p=2$ and \eqref{L2-varphi}, we obtain  
\begin{align*} 
\rho^{-\frac12} h^{-\frac12} 
\|\nabla (v-v_h)\|_{L^2(\Lambda_h\cap S_{4\rho}(x_0))} 
&\le
C\rho^{-\frac12} h^{-\frac12} h 
\|v\|_{H^2(\Omega)} \nonumber \\
&\le
C\rho^{-\frac12} h^{-\frac12} h 
\|\varphi\|_{L^2(\Omega)} \le C,
\end{align*}
where we have used $\rho\ge h$ and $\|\varphi\|_{L^2(\Omega)}\le 1$ in deriving the last inequality. 
Substituting the last inequality into \eqref{remains-1} yields 
\begin{align}\label{remains-2}
&\rho^{-\frac32} h^{-1} \|\nabla (v-v_h)\|_{L^1(\Lambda_h)} 
\le
C\rho^{-\frac32} h^{-\frac12} \sum_{j=0}^J d_j
\|\nabla (v-v_h)\|_{L^2(A_j)} 
+C . 
\end{align}
Now, we use the following interior energy error estimate (proved in \cite[Theorem 5.1]{1974-Nitsche-Schatz}, also see \cite[Lemma 2.1 (i)]{Schatz-1980}): 
\begin{align}\label{interior-H1}
\|\nabla (v-v_h)\|_{L^2(A_j)} \le
C\|\nabla (v-I_hv)\|_{L^2(A_j^1)} 
&+Cd_j^{-1}\|v-I_hv\|_{L^2(A_j^1)}\nonumber \\
&+Cd_j^{-1}\|v-v_h\|_{L^2(A_j^1)},
\end{align}
where $I_h$ denotes the nodal interpolant. Using the approximation theory, we obtain 
\begin{align}\label{interior-H2}
\|\nabla (v-v_h)\|_{L^2(A_j)}
&\le
\big(Ch +Ch^{2}d_j^{-1}\big) 
\|v\|_{H^{2}(A^2_j)}
+Cd_j^{-1}\|v-v_h\|_{L^2(A_j^1)} \nonumber \\
&\le
Ch d_j^{\frac12-\frac{3}{p}}\|v\|_{W^{1,p}(A_j^3)} 
+Cd_j^{-1}\|v-v_h\|_{L^2(A_j^1)} 
\quad\mbox{for}\,\,\, \mbox{$\frac65$}<p<2, 
\end{align}
where we have used $d_j\ge h$ and the following inequality in deriving the last inequality:  
\begin{align}\label{interior-harmonic}
\|v\|_{H^{2}(A^2_j)}\le
Cd_j^{\frac12-\frac{3}{p}}\|v\|_{W^{1,p}(A_j^3)} 
\quad \mbox{for}\,\,\, \mbox{$\frac65$}<p<2.
\end{align}
The inequality above follows from Lemma \ref{lemma: Caccipoli}, the H\"older inequality and Sobolev embedding, i.e. 
$$
\begin{aligned}
\|v\|_{H^2(A^2_j)} 
&\le Cd_j^{-2}\|v\|_{L^{2}(A^{3}_j)}\\
&\le Cd_j^{-2+\frac32-\frac{3}{q}}\|v\|_{L^{q}(A^{3}_j)} &&\mbox{if $q>2$}\\
&\le Cd_j^{\frac12-\frac{3}{p}}\|v\|_{W^{1,p}(A^{3}_j)}  
&&
\mbox{for $\frac3q=\frac3p-1$ and $\mbox{$\frac65$}<p<2$ (so that $q>2$)}
.
\end{aligned}
$$
This proves that \eqref{interior-H1} holds for $\mbox{$\frac65$}<p<2$. 

By applying Lemma \ref{W1p-d} to \eqref{interior-H1} with $p=\frac32$, we obtain 
\begin{align}\label{interior-2}
&\|\nabla (v-v_h)\|_{L^2(A_j)} \nonumber \\ 
&\le Ch 
d_j^{-\frac32}\rho 
\|\varphi\|_{L^{\frac32}(S_\rho(x_0))} 
+Cd_j^{-1}\|v-v_h\|_{L^2(A_j^1)} \nonumber \\
&\le Ch d_j^{-\frac32}\rho^{\frac32}
+Cd_j^{-1}\|v-v_h\|_{L^2(A_j^1)},
\end{align}
where the last inequality is due to the following H\"older inequality:
$$
\|\varphi\|_{L^{\frac32}(S_\rho(x_0))} 
\le 
C\rho^{\frac12}
\|\varphi\|_{L^{2}(S_\rho(x_0))} 
\quad
\mbox{with}\quad
\|\varphi\|_{L^{2}(S_\rho(x_0))} \le 1.
$$
From \eqref{interior-2} we see that 
\begin{align}\label{interior-3}
d_j\|\nabla (v-v_h)\|_{L^2(A_j)} \le
C\rho^{\frac32} h^{\frac12}  \bigg(\frac{h}{d_j}\bigg)^{\frac12} 
+C \|v-v_h\|_{L^2(A_j^1)}.
\end{align}
Then, substituting \eqref{interior-3} into \eqref{remains-2}, we have 
\begin{align}\label{remains-4}
&\rho^{-\frac32} h^{-1} \|\nabla (v-v_h)\|_{L^1(\Lambda_h)} 
\nonumber \\ 
&\le
C\sum_{j=0}^J \bigg(\frac{h}{d_j}\bigg)^{\frac12}   
+
C\rho^{-\frac32} h^{-\frac12}\sum_{j=0}^J \|v-v_h\|_{L^2(A_j^1)} \nonumber \\
&\le C+C\rho^{-\frac32} h^{-\frac12}\sum_{j=0}^J
\|v-v_h\|_{L^2(A_j^1)} .
\end{align}

It remains to estimate $\sum_{j=0}^J\|v-v_h\|_{L^2(A_j^1)}$. To this end, we let $\chi$ be a smooth cut-off function satisfying 
$$
\chi=1
\,\,\,\mbox{on}\,\,\,
A_j^1
\quad\mbox{and}\quad
\chi=0
\,\,\,\mbox{outside}\,\,\,
A_j^2.
$$
Then
\begin{align}\label{v-vh-L6}
\|v-v_h\|_{L^6(A_j^1)}
&\le
\|\chi(v-v_h)\|_{L^6(\Omega)} \nonumber \\
&\le
\|\chi(v-v_h)\|_{H^1(\Omega)}
\quad\mbox{(Sobolev embedding $H^1(\Omega)\hookrightarrow L^6(\Omega)$)} \nonumber \\
&\le
\|\nabla(v-v_h)\|_{L^2(A_j^2)}
+Cd_j^{-1}\|v-v_h\|_{L^2(A_j^2)} . 
\end{align}
By using \eqref{v-vh-L6} and the interpolation inequality (for $1<p<2$) 
\begin{align}\label{Lp-interpl}
\|v-v_h\|_{L^2(A_j^1)}
\le
\|v-v_h\|_{L^p(A_j^1)}^{1-\theta}
\|v-v_h\|_{L^6(A_j^1)}^\theta
\quad\mbox{with}\,\,\,
\frac{1}{2}
=\frac{1-\theta}{p}+\frac{\theta}{6} ,
\end{align}
we obtain
\begin{align*}
&\|v-v_h\|_{L^2(A_j^1)} \\
&\le
\|v-v_h\|_{L^p(A_j^1)}^{1-\theta}
\big(\|\nabla(v-v_h)\|_{L^2(A_j^2)}
+Cd_j^{-1}\|v-v_h\|_{L^2(A_j^2)}\big)^{\theta} \\
&=
(\varepsilon^{-\frac{\theta}{1-\theta}}\|v-v_h\|_{L^p(A_j^1)})^{1-\theta}
\big(\varepsilon\|\nabla(v-v_h)\|_{L^2(A_j^2)}
+C\varepsilon d_j^{-1}\|v-v_h\|_{L^2(A_j^2)}\big)^{\theta} \\
&\le
\varepsilon^{-\frac{\theta}{1-\theta}}\|v-v_h\|_{L^p(A_j^1)} 
+\varepsilon\|\nabla(v-v_h)\|_{L^2(A_j^2)}
+C\varepsilon d_j^{-1}\|v-v_h\|_{L^2(A_j^2)} ,
\end{align*}
where $\varepsilon$ can be an arbitrary positive number.
By choosing $\varepsilon=d_j(\rho/d_j)^{\sigma}$ with $\sigma\in(0,1)$, we obtain
\begin{align}\label{Lp-interpl2}
\|v-v_h\|_{L^2(A_j^1)} 
&\le
\bigg(\frac{\rho}{d_j}\bigg)^{-\frac{\theta\sigma}{1-\theta}}
d_j^{-\frac{\theta}{1-\theta}}\|v-v_h\|_{L^p(A_j^1)} \\
&\quad\,  
+\bigg(\frac{\rho}{d_j}\bigg)^{\sigma}
\big(d_j\|\nabla(v-v_h)\|_{L^2(A_j^2)}
+C\|v-v_h\|_{L^2(A_j^2)} \big) . 
\end{align}
Hence,
\begin{align}\label{Lp-interpl3}
&\rho^{-\frac32} h^{-\frac12}\sum_{j=0}^J\|v-v_h\|_{L^2(A_j^1)} \notag \\
&\le 
C\rho^{-\frac32} h^{-\frac12} 
\sum_{j=0}^J \bigg(\frac{\rho}{d_j}\bigg)^{-\frac{\theta\sigma}{1-\theta}}
d_j^{-\frac{\theta}{1-\theta}}\|v-v_h\|_{L^p(A_j^1)} \notag \\
&\quad\, 
+
C\rho^{-\frac32} h^{-\frac12}\sum_{j=0}^J 
\bigg(\frac{\rho}{d_j}\bigg)^{\sigma}
\big(d_j\|\nabla(v-v_h)\|_{L^2(A_j^2)}
+C\|v-v_h\|_{L^2(A_j^2)} \big) \notag \\
&\le 
C\rho^{-\frac32} h^{-\frac12}\sum_{j=0}^J \bigg(\frac{\rho}{d_j}\bigg)^{-\frac{\theta\sigma}{1-\theta}}
d_j^{-\frac{\theta}{1-\theta}}\|v-v_h\|_{L^p(A_j^1)} \notag \\
&\quad\, 
+
C\rho^{-\frac32} h^{-\frac12}\sum_{j=0}^J \bigg(\frac{\rho}{d_j}\bigg)^{\sigma}\|v-v_h\|_{L^2(A_j^2)},
\end{align}
where we have used \eqref{interior-3} in deriving the last inequality. 
Note that 
\begin{align*}
&\sum_{j=0}^J \bigg(\frac{\rho}{d_j}\bigg)^{\sigma}\|v-v_h\|_{L^2(A_j^2)} \\
&\le
C\bigg(\frac{\rho}{d_j}\bigg)^{\sigma}\|v-v_h\|_{L^2(S_{8\rho}(x_0))} 
+ 
2\sum_{j=0}^J \bigg(\frac{\rho}{d_j}\bigg)^{\sigma}\|v-v_h\|_{L^2(A_j^1)} . 
\end{align*}
Combining the last two estimates, we obtain
\begin{align*}
&\rho^{-\frac32} h^{-\frac12} \sum_{j=0}^J\|v-v_h\|_{L^2(A_j^1)} \\
&\le 
C\rho^{-\frac32} h^{-\frac12}\sum_{j=0}^J \bigg(\frac{\rho}{d_j}\bigg)^{-\frac{\theta\sigma}{1-\theta}}
d_j^{-\frac{\theta}{1-\theta}}\|v-v_h\|_{L^p(A_j^1)} \\
&\quad\, 
+ C\rho^{-\frac32} h^{-\frac12}\bigg(\frac{\rho}{d_j}\bigg)^{\sigma}\|v-v_h\|_{L^2(S_{8\rho}(x_0))}
+
C\rho^{-\frac32} h^{-\frac12}\sum_{j=0}^J \bigg(\frac{\rho}{d_j}\bigg)^{\sigma}\|v-v_h\|_{L^2(A_j^1)}.
\end{align*}
If $d_j\ge \kappa\rho$ for sufficiently large constant $\kappa$, then the last term can be absorbed by the left side. Hence, we have 
\begin{align}\label{L2-Lp}
\sum_{j=0}^J\rho^{-\frac32} h^{-\frac12}\|v-v_h\|_{L^2(A_j^1)} 
&\le 
\sum_{j=0}^J C\rho^{-\frac32} h^{-\frac12}
\bigg(\frac{\rho}{d_j}\bigg)^{-\frac{\theta\sigma}{1-\theta}}
d_j^{-\frac{\theta}{1-\theta}}\|v-v_h\|_{L^p(A_j^1)} \nonumber \\
&\quad\, 
+ C\rho^{-\frac32} h^{-\frac12}\bigg(\frac{\rho}{d_j}\bigg)^{\sigma}\|v-v_h\|_{L^2(S_{8\rho}(x_0))} . 
\end{align}

It remains to estimate $\|v-v_h\|_{L^p(A_j^1)}$ and $\|v-v_h\|_{L^2(S_{8\rho}(x_0))}$. To this end, we let $\psi\in C^\infty_0(A_j^1)$ be a function satisfying 
\begin{align}\label{Lp-interpl4}
\|v-v_h\|_{L^p(A_j^1)} 
&\le 2 (v-v_h,\psi) 
\quad\mbox{and}\quad 
\|\psi\|_{L^q(A_j^1)}\le 1 ,\quad\mbox{with}\quad
\frac1p+\frac1q=1.
\end{align}
Let $w\in H^1_0(\Omega)$ be the solution of 
\begin{align*} 
\left\{
\begin{aligned}
-\Delta w&=\psi &&\mbox{in}\,\,\,\Omega, \\
w&=0 &&\mbox{on}\,\,\,\partial\Omega, 
\end{aligned}
\right.
\end{align*}
Then using Lemma \ref{Lemma:Ritz-error-W1p} and Lemma \ref{lemma: W2p regularity}, we obtain
\begin{align*} 
(v-v_h,\psi) 
&=(\nabla(v-v_h),\nabla w) \\ 
&=(\nabla(v-v_h),\nabla (w-I_hw)) \\ 
&\le \|\nabla(v-v_h)\|_{L^p(\Omega)} 
\|\nabla(w-I_hw)\|_{L^q(\Omega)} \\ 
&\le Ch^{2} 
\|v\|_{W^{2,p}(\Omega)} 
\|w\|_{W^{2,q}(\Omega)} \\
&\le Ch^{2} 
\|\varphi\|_{L^p(\Omega)} 
\|\psi\|_{L^q(\Omega)} \\
&\le Ch^{2} 
\|\varphi\|_{L^p(S_\rho(x_0))}  \\
&\le Ch^{2}\rho^{\frac3p-\frac32} 
\|\varphi\|_{L^2(S_\rho(x_0))} 
\|\psi\|_{L^q(A_j^1)} \\
&\le Ch^{2}\rho^{\frac3p-\frac32} ,
\end{align*} 
where we have used $\|\varphi\|_{L^2(S_\rho(x_0))} \le 1$ and $\|\psi\|_{L^q(A_j^1)} \le 1$ in deriving the last inequalities. 
This implies
\begin{align}\label{Lp-interpl5}
\|v-v_h\|_{L^p(A_j^1)} 
\le
Ch^{2}\rho^{\frac3p-\frac32}
\quad\mbox{and}\quad
\|v-v_h\|_{L^2(S_{8\rho}(x_0))} 
\le
Ch^{2} .
\end{align}
By substituting these estimates into \eqref{L2-Lp}, we obtain
\begin{align}\label{L2-Lp-2}
\sum_{j=0}^J\rho^{-\frac32} h^{-\frac12}\|v-v_h\|_{L^2(A_j^1)} 
&\le 
\sum_{j=0}^J C \bigg(\frac{h}{\rho}\bigg)^{\frac32} 
\bigg(\frac{\rho}{d_j}\bigg)^{\frac3p-\frac32-\frac{\theta\sigma}{1-\theta}}
+C  .
\end{align}
Since $p<2$, by choosing sufficiently small $\sigma$ we have $\frac3p-\frac32-\frac{\theta\sigma}{1-\theta}>0$ and therefore 
\begin{align}\label{L2-Lp-2}
\sum_{j=0}^J\rho^{-\frac32} h^{-\frac12}\|v-v_h\|_{L^2(A_j^1)} 
&\le 
\sum_{j=0}^J C \bigg(\frac{h}{\rho}\bigg)^{\frac32} 
\bigg(\frac{\rho}{d_j}\bigg)^{\frac3p-\frac32-\frac{\theta\sigma}{1-\theta}} + C
\le C .
\end{align}
Then, substituting this into \eqref{remains-4}, we obtain 
\begin{align} 
&\rho^{-\frac32} h^{-1} \|\nabla (v-v_h)\|_{L^1(\Lambda_h)} 
\le C.
\end{align}
This proves the desired result 
for sufficiently small mesh size $h\le h_0$, as explained in the end of section \ref{sec: preliminary}. 

For $h\ge h_0$, we denote by $\widetilde{g}_h\in S_h$ the finite element function satisfying $\widetilde g_h=u_h$ on $\partial\Omega$ and  $\widetilde g_h=0$ at the interior nodes  of the domain $\Omega$. Naturally,  
$$
\|\widetilde g_h\|_{L^\infty(\Omega)}\le \|u_h\|_{L^\infty(\partial\Omega)} .
$$
Since $\chi_h=u_h-\widetilde g_h\in \mathring S_h$, from \eqref{discrete-harmonic}, we have
$$
(\nabla u_h,(\nabla( u_h-\widetilde g_h))=0
$$
and as a result
$$
\begin{aligned}
\|\nabla (u_h-\widetilde g_h)\|^2_{L^2(\Omega)}
=(\nabla(u_h-\widetilde g_h),\nabla(u_h-\widetilde g_h))&=-(\nabla \widetilde g_h,\nabla(u_h-\widetilde g_h))\\
&\le
C\|\nabla\widetilde g_h\|_{L^2(\Omega)}\|\nabla(u_h-\widetilde g_h)\|_{L^2(\Omega)}.
\end{aligned}
$$
Thus, using the inverse inequality and that $h\ge h_0$, we have
$$
\begin{aligned}
\|\nabla (u_h-\widetilde g_h)\|_{L^2(\Omega)}
\le
C\|\nabla \widetilde g_h\|_{L^2(\Omega)}
\le
Ch^{-1}\|\widetilde g_h\|_{L^2(\Omega)}
&\le
Ch_0^{-1}\|\widetilde g_h\|_{L^\infty(\Omega)} \\
&\le Ch_0^{-1}\|u_h\|_{L^\infty(\partial\Omega)}.
\end{aligned}
$$
By using the inverse inequality and the above estimate, we also have 
\begin{align*}
\|u_h-\widetilde g_h\|_{L^\infty(\Omega)}
&\le
Ch^{-\frac32}\|u_h-\widetilde g_h\|_{L^2(\Omega)} \\
&\le
Ch^{-\frac32}\|\nabla (u_h-\widetilde g_h)\|_{L^2(\Omega)} \\
&\le Ch_0^{-\frac52}\|u_h\|_{L^\infty(\partial\Omega)}
.
\end{align*}
By the triangle inequality, this proves
 $$
 \|u_h\|_{L^\infty(\Omega)} \le  \|\widetilde g_h\|+\|u_h-\widetilde g_h\|_{L^\infty(\Omega)}\le C \|u_h\|_{L^\infty(\partial\Omega)}
 $$
  for $h\ge h_0$. 

Combining the two cases $h\le h_0$ and $h\ge h_0$, we obtain the desired result of Theorem \ref{THM1}.

\section{Application to the Ritz projection}\label{sec: stability Ritz}

In this section, we adopt Schatz's argument to prove the maximum-norm stability of the Ritz projection. This argument uses the weak maximum principle established above to remove a logarithmic factor for  finite elements of degree $r\ge 2$ in convex polyhedral domains under the following assumption:
\begin{enumerate}
\item[(A)]
The tetrahedral partition of $\Omega$ can be extended to a larger convex domain $\widetilde{\Omega}$ quasi-uniformly, with $\Omega\subset\subset\widetilde{\Omega}$. 
\end{enumerate}
The logarithmic factor has been removed in previous articles only for $r\geq 2$ on smooth and two-dimensional polygonal domains. 

For any function $u\in H^1_0(\Omega)$, we denote by $R_hu\in \mathring S_h$ the Ritz projection of $u$, defined by  
\begin{equation}\label{Ritz-u-uh}
(\nabla (u-R_hu),\nabla \chi_h)=0\quad\forall\,\chi_h\in\mathring S_h. 
\end{equation}
\begin{theorem}\label{THM-Ritz}
Under assumption {\rm(A)}, for finite elements of degree $r\geq 2$ the Ritz projection satisfies  
\begin{equation}\label{Ritz-Linfty}
\|R_hu\|_{L^\infty(\Omega)} \le C\|u\|_{L^\infty(\Omega)}
\quad\forall\, u\in H^1_0(\Omega)\cap C(\overline\Omega). 
\end{equation}
\end{theorem}
\begin{proof}
Let $\widetilde u$ be the zero extension of $u$ to the larger domain $\widetilde\Omega$. 
Let $\mathring S_h(\widetilde\Omega)$ be the finite element space subject to the tetrahedral partition of $\widetilde\Omega$ (with zero boundary values), and let $\widetilde{u}_h$ be the Ritz projection of $\widetilde u$ in the domain $\widetilde\Omega$, i.e. 
\begin{equation}\label{Ritz-u-uh}
\int_{\widetilde\Omega} \nabla (\widetilde{u}-\widetilde{u}_h)\cdot\nabla \chi_h \d x=0\quad\forall\,\chi_h\in\mathring S_h(\widetilde\Omega). 
\end{equation}
Since $u=\widetilde{u}$ on $\Omega$, it follows that 
\begin{align}\label{E1+E2}
\|u-u_h\|_{L^\infty(\Omega)}
&=\|\widetilde{u}-u_h\|_{L^\infty(\Omega)} \\
&\le \|\widetilde{u}-\widetilde{u}_h\|_{L^\infty(\Omega)}+\|\widetilde{u}_h-u_h\|_{L^\infty(\Omega)} 
\\
:\!&=E_1+E_2.
\end{align}

By using \cite[Theorem 5.1]{Schatz-Wahlbin-1977} (which requires $r \geq 2$ to remove a logarithmic factor and $h$ sufficiently small, say $h\le h_*$), we have
\begin{align}\label{E1}
E_1\le C\|\widetilde{u}-I_h\widetilde{u}\|_{L^\infty(\Omega')}+C\|\widetilde{u}-\widetilde{u}_h\|_{L^2(\Omega')},
\end{align}
where $\Omega'$ is some intermediate domain satisfying $\Omega\subset\subset\Omega'\subset\subset\widetilde{\Omega}$. Since the Lagrange interpolation operator $I_h$ is stable in the $L^\infty$ norm on $C(\overline\Omega)$, it follows that  
\begin{align}\label{E1-1}
\|\widetilde{u}-I_h\widetilde{u}\|_{L^\infty(\Omega')}\le C\|\widetilde{u}\|_{L^\infty(\Omega)}=C\|{u}\|_{L^\infty(\Omega)} . 
\end{align}
To estimate $\|\widetilde{u}-\widetilde{u}_h\|_{L^2(\Omega')}$, we use a duality argument. Thus,
$$
\|\widetilde{u}-\widetilde{u}_h\|_{L^2(\Omega')}\le\|\widetilde{u}-\widetilde{u}_h\|_{L^2(\widetilde{\Omega})}  = \sup_{{\widetilde{\varphi}\in C^\infty_0(\widetilde{\Omega})} \atop {\|\widetilde{\varphi}\|_{L^2(\widetilde{\Omega})}\le 1}}
\int_{\widetilde\Omega}
(\widetilde{u}-\widetilde{u}_h) \,\widetilde{\varphi} \,\d x . 
$$
In particular, there exists a $\widetilde\varphi \in C^\infty_0(\widetilde{\Omega})$ satisfying 
\begin{align}\label{u-uh-L2-0}
\|\widetilde{\varphi}\|_{L^2(\widetilde{\Omega})}\le 1 
\quad\mbox{and}\quad 
\|\widetilde{u}-\widetilde{u}_h\|_{L^2(\widetilde{\Omega})}  
\le 2 \int_{\widetilde\Omega}
(\widetilde{u}-\widetilde{u}_h) \,\widetilde{\varphi} \,\d x . 
\end{align}
For this $\widetilde\varphi$ we define $\widetilde\psi\in H^1_0(\Omega)$ to be the weak solution of 
\begin{equation}\label{eq: Laplace psi}
\left\{
\begin{aligned}
-\Delta \widetilde{\psi} &= \widetilde\varphi && \mbox{in}\,\,\, \widetilde{\Omega} , \\
    \widetilde{\psi} &= 0 && \mbox{on}\,\,\, \partial\widetilde{\Omega} , 
\end{aligned}
\right. 
\end{equation}
and denote by $\widetilde{\psi}_h\in\mathring S_h(\widetilde\Omega)$ the Ritz projection of $\widetilde\psi$ in $\widetilde\Omega$, i.e. 
\begin{equation}\label{FEM: Laplace psi}
\int_{\widetilde\Omega}\nabla(\widetilde{\psi}-\widetilde{\psi}_h)\cdot\nabla\widetilde{\chi}_h \, \d x=0\quad\forall\,\widetilde{\chi}_h\in \mathring S_h(\widetilde\Omega). 
\end{equation}
If we denote by $\widetilde{\frak{T}}$ the set of tetrahedra in the partition of $\widetilde\Omega$, then testing \eqref{eq: Laplace psi} by $\widetilde{u}-\widetilde{u}_h$ yields 
\begin{align}\label{u-uh-L2}
\begin{aligned}
\int_{\widetilde\Omega}
(\widetilde{u}-\widetilde{u}_h) \,\widetilde{\varphi} \,\d x 
&=\int_{\widetilde\Omega}
\nabla(\widetilde{u}-\widetilde{u}_h) \cdot \nabla\widetilde{\psi} \,\d x \hspace{20pt} \mbox{(here we use integration by parts)} \\ 
&=\int_{\widetilde\Omega}
\nabla(\widetilde{u}-\widetilde{u}_h) \cdot \nabla(\widetilde{\psi}-\widetilde{\psi}_h) \d x 
\quad \mbox{(here we use \eqref{Ritz-u-uh})} \\ 
&=\int_{\widetilde\Omega}
\nabla \widetilde{u}  \cdot \nabla(\widetilde{\psi}-\widetilde{\psi}_h) \d x 
\hspace{41pt}  \mbox{(here we use \eqref{FEM: Laplace psi})} \\
&=\sum_{\tau \in \widetilde{\frak{T}}}
\int_{\tau}
\nabla \widetilde{u}  \cdot \nabla(\widetilde{\psi}-\widetilde{\psi}_h) \d x \\ 
&=- \sum_{\tau\in \widetilde{\frak{T}}} \int_{\tau} 
\widetilde{u} \, \Delta(\widetilde{\psi}-\widetilde{\psi}_h) \,\d x 
+ 
\int_{\partial\tau} \widetilde{u}\, \partial_n (\widetilde{\psi}-\widetilde{\psi}_h) \, \d s \\
&\le C \|\widetilde{u}\|_{L^\infty(\widetilde{\Omega})} 
\sum_{\tau\in \widetilde{\frak{T}}} 
\Big(\|\Delta(\widetilde{\psi}-\widetilde{\psi}_h)\|_{L^1(\tau)} 
+ \|\partial_n(\widetilde{\psi}-\widetilde{\psi}_h)\|_{L^1(\partial\tau)}
\Big) \\
&\le C\|u\|_{L^\infty(\Omega)}\bigg(h^{-1}\|\nabla(\widetilde{\psi}-\widetilde{\psi}_h)\|_{L^1(\widetilde\Omega)}+\sum_{\tau\in \widetilde{\frak{T}}}\|\widetilde{\psi}-\widetilde{\psi}_h\|_{W^{2,1}(\tau)}\bigg) ,
\end{aligned}
\end{align}
where in the last step we have used $\|\widetilde{u}\|_{L^\infty(\widetilde{\Omega})} = \|u\|_{L^\infty(\Omega)}$ and the trace inequality 
$$
\|\partial_n(\widetilde{\psi}-\widetilde{\psi}_h)\|_{L^1(\partial\tau)}\le Ch^{-1}\|\nabla(\widetilde{\psi}-\widetilde{\psi}_h)\|_{L^1(\tau)}+C\|\widetilde{\psi}-\widetilde{\psi}_h\|_{W^{2,1}(\tau)}.
$$
By using a priori energy estimate and $H^2$ regularity, we have 
\begin{align}\label{W11-psi}
\|\nabla(\widetilde{\psi}-\widetilde{\psi}_h)\|_{L^1(\widetilde{\Omega})}
&\le \|\nabla(\widetilde{\psi}-\widetilde{\psi}_h)\|_{L^2(\widetilde{\Omega})} \nonumber \\
&\le \|\nabla(\widetilde{\psi}-I_h\widetilde{\psi})\|_{L^2(\widetilde{\Omega})} \nonumber \\ 
&\le Ch\|\widetilde{\psi}\|_{H^2(\widetilde{\Omega})} \nonumber \\
&\le Ch\|\widetilde{\varphi}\|_{L^2(\widetilde{\Omega})} 
\le Ch.
\end{align}
Let $\tilde{I}_h$ be the Scott-Zhang interpolant. Then by the triangle and inverse inequalities, we have 
$$
\begin{aligned}
&\sum_{\tau\in \widetilde{\frak{T}}}\|\widetilde{\psi}-\widetilde{\psi}_h\|_{W^{2,1}(\tau)} \\ 
&\le C\sum_{\tau\in \widetilde{\frak{T}}}\bigg(\|\widetilde{\psi}-\tilde{I}_h\widetilde{\psi}\|_{W^{2,1}(\tau)}+\|\tilde{I}_h\widetilde{\psi}-\widetilde{\psi}_h\|_{W^{2,1}(\tau)}\bigg)\\
&\le C\bigg(\sum_{\tau\in \widetilde{\frak{T}}}\|\widetilde{\psi}-\tilde{I}_h\widetilde{\psi}\|_{W^{2,1}(\tau)}+h^{-1}\|\tilde{I}_h\widetilde{\psi}-\widetilde{\psi}_h\|_{W^{1,1}(\widetilde{\Omega})}\bigg)\\
&\le C\bigg(\sum_{\tau\in \widetilde{\frak{T}}} \|\widetilde{\psi}-\tilde{I}_h\widetilde{\psi}\|_{W^{2,1}(\tau)}+h^{-1}\|\widetilde{\psi}-\tilde{I}_h\widetilde{\psi}\|_{W^{1,1}(\widetilde{\Omega})}+h^{-1}\|\widetilde{\psi}-\widetilde{\psi}_h\|_{W^{1,1}(\widetilde{\Omega})}\bigg) . 
\end{aligned}
$$
Similarly as \eqref{W11-psi}, we can prove the following estimate: 
$$
h^{-1}\|\widetilde{\psi}-\tilde{I}_h\widetilde{\psi}\|_{W^{1,1}(\widetilde{\Omega})}+h^{-1}\|\widetilde{\psi}-\widetilde{\psi}_h\|_{W^{1,1}(\widetilde{\Omega})}\le C ,
$$
and by using the properties of $\tilde{I}_h$ (cf. \cite[Theorem 4.8.3.8]{Brenner_Scott}), 
$$
\sum_{\tau\in \widetilde{\frak{T}}} \|\widetilde{\psi}-\tilde{I}_h\widetilde{\psi}\|_{W^{2,1}(\tau)}\le C
\sum_{\tau\in \widetilde{\frak{T}}} \|\widetilde{\psi}\|_{W^{2,1}(\tau)}
\le C\|\widetilde{\psi}\|_{H^2(\widetilde{\Omega})}\le C\|\widetilde{\varphi}\|_{L^2(\widetilde{\Omega})}\le C.
$$
Now we substitute these estimates into \eqref{u-uh-L2}. This yields  
\begin{align}\label{E1-2}
\|\widetilde{u}-\widetilde{u}_h\|_{L^2(\widetilde{\Omega})} 
\le C \|u\|_{L^\infty(\Omega)} . 
\end{align}
Then, by substituting \eqref{E1-1} and \eqref{E1-2} into \eqref{E1}, we obtain 
\begin{align}\label{E1-estimate}
E_1 \le C\|u\|_{L^\infty(\Omega)} .
\end{align}
To estimate $E_2$, we use the fact that $\widetilde{u}_h-u_h$ is discrete harmonic in $\Omega$, i.e. 
$$
\int_\Omega 
\nabla(\widetilde{u}_h-u_h)\cdot \nabla\chi_h\,\d x 
= 
\int_\Omega \nabla(\widetilde{u}-u) \cdot \nabla\chi_h\,\d x = 
0 \quad\forall\,\chi_h\in\mathring S_h(\Omega). 
$$
Thus, by the weak discrete maximum principle proved in Theorem \ref{THM1} and using the fact that $u_h=0$ and $\widetilde{u}=0$ on $\partial \Omega$, we have 
\begin{align}\label{E2-estimate}
\begin{aligned}
E_2
&=\|\widetilde{u}_h-u_h\|_{L^\infty(\Omega)} \\
&\le C\|\widetilde{u}_h-u_h\|_{L^\infty(\partial\Omega)} \\
&=C\|\widetilde{u}_h\|_{L^\infty(\partial\Omega)} 
&&\mbox{(use $u_h=0$ on $\partial\Omega$)} \\
&=C\|\widetilde{u}_h-\widetilde{u}\|_{L^\infty(\partial\Omega)}
&&\mbox{(use $\widetilde{u}=0$ on $\partial\Omega$)} \\
&\le C
\|\widetilde{u}_h-\widetilde{u}\|_{L^\infty(\Omega)} \\
&= E_1 ,
\end{aligned}
\end{align}
which has already been estimated. 
Hence, substituting \eqref{E1-estimate} and \eqref{E2-estimate} into 
\eqref{E1+E2}, we obtain 
\begin{align}
\|u-u_h\|_{L^\infty(\Omega)} 
\le C\|u\|_{L^\infty(\Omega)} . 
\end{align}
This completes the proof of Theorem \ref{THM-Ritz} in the case $h\le h_*$ for some positive constant $h_*$.

If $h\ge h_*$ then we pick up a point $x_0\in\bar{\tau}_0$ (in some tetrahedron $\tau_0$) satisfying $|u_h(x_0)|=\|u_h\|_{L^\infty(\Omega)}$. For such $x_0$ we define a regularized Green's function $G$ as the solution of 
\begin{equation}\label{eq: regularized Green}
\begin{aligned}
-\Delta G(x) &= \tilde{\delta}(x), && x\in \Omega,\;  \\
    G(x) &= 0,    && x \in \partial\Omega,
\end{aligned}
\end{equation}
where $\tilde{\delta}\in C^3(\overline\Omega)$ is the regularized Delta function concentrated at $x_0$, satisfying  
${\rm supp}(\tilde{\delta})\subset \bar\tau_0$ and 
\begin{align*}
&\int_{\Omega}\chi_h \widetilde\delta \d x = \chi_h(x_0) ,
\quad\forall\,\chi_h\in \mathring S_h , \\[5pt]
&\|\widetilde\delta \|_{W^{l,p}}
\leq K h^{-l-3(1-1/p)}
\quad\mbox{for}\,\,\,1\leq p\leq\infty,
\,\,\, l=0,1,2,3 . \label{reg-Delta-est}
\end{align*} 
The construction of the function $\widetilde\delta $ can be found in \cite[Lemma 2.2]{ThomeeWahlbin2000}. In particular, the construction of $\widetilde\delta $ can be done in any tetrahedron for the arbitrary mesh size $h$.  

We define $G_h=R_hG\in \mathring S_h$, i.e., 
\begin{equation}\label{eq:g_h}
(\nabla G_h,\nabla \chi_h) = (\tilde{\delta},\chi_h) \quad \forall \chi_h \in \mathring S_h .
\end{equation}
The finite element function $G_h$ defined by the equation above satisfies the following standard energy estimate:
$$
\| G_h\|_{H^{1}(\Omega)} \le C\|\tilde \delta \|_{L^2(\Omega)} . 
$$
Then using the Galerkin orthogonality, integration by parts, we obtain
\begin{equation}\label{eq: initial estimate for u_h(x_0)}
\begin{aligned}
u_h(x_0)
&=(\nabla u_h,\nabla G_h)=(\nabla u,\nabla G_h)=\sum_{\tau \in \mathfrak{T}}\left[ (u,\partial_n G_{h})_{\partial\tau}+(u,-\Delta G_h)_\tau \right]\\
&\le \|u\|_{L^\infty(\Omega)}\sum_{\tau \in \mathfrak{T}}\left(\|\partial_n G_{h}\|_{L^1(\partial\tau)}+\|\Delta G_h\|_{L^1(\tau)}\right).
\end{aligned}
\end{equation}
Now, for $h\geq h_*$, using the trace and inverse inequality we have
\begin{align*}
\sum_{\tau \in \mathfrak{T}}\left[\|\partial_n G_{h}\|_{L^1(\partial\tau)}+\|\Delta G_h\|_{L^1(\tau)}\right]&\le Ch^{-1}\sum_{\tau \in \mathfrak{T}}\|\nabla G_h\|_{L^1(\tau)}\\ &
\le Ch^{-1}\| G_h\|_{W^{1,1}(\Omega)}\\
&\le Ch^{-1}\| G_h\|_{H^{1}(\Omega)}\\
&\le Ch^{-1}\|\tilde \delta \|_{L^2(\Omega)}\le Ch_*^{-5/2},
\end{align*}
since $\| \tilde \delta \|_{L^2(\Omega)}\le Ch^{-3/2}$ and $h\ge h_*$.

Combining the two cases $h\le h_*$ and $h\ge h_*$, we obtain the desired result of Theorem \ref{THM-Ritz}.
\end{proof}

\section{Conclusion}

In this article, we have proved the weak maximum principle of finite element method (Theorem \ref{THM1}). The main difference between the current proof and the proof in \cite{Schatz-1980} for two-dimensional polygons is that we have used $L^p$ estimates in place of some $L^2$ estimates in section \ref{sec: main proof}, including \eqref{interior-H1}, \eqref{interior-harmonic}, \eqref{Lp-interpl}, \eqref{Lp-interpl2}, \eqref{Lp-interpl3}, \eqref{Lp-interpl4} and \eqref{Lp-interpl5}. As an application of the weak maximum principle of finite element methods, we have presented an $L^\infty$-stability of Ritz projection (Theorem \ref{THM-Ritz}) by utilizing the argument in \cite[Theorem 5.1]{Schatz-1980}. 
  
\section*{Acknowledgement}

We thank the anonymous referees for the valuable comments and suggestions.

\end{document}